\documentclass[12pt]{amsart}
\usepackage{amsfonts}
\usepackage{amsfonts}
\usepackage{bbm}
\usepackage{amsfonts,amssymb,amsmath,amsthm}
\usepackage{url}
\usepackage{enumerate}
\usepackage{bbm}
\usepackage[all]{xy}

\newtheorem{thrm}{Theorem}[section]
\newtheorem{lem}[thrm]{Lemma}
\newtheorem{prop}[thrm]{Proposition}

\theoremstyle{definition}
\newtheorem{definition}[thrm]{Definition}
\newtheorem{remark}[thrm]{Remark}
\numberwithin{equation}{section}

\author{George A. Elliott}
\address{Department of Mathematics, University of Toronto, Toronto, Ontario, Canada M5S 2E4}
\email{elliott@math.toronto.edu}
\author{Zhiqiang Li}
\address{College of Mathematics and Information Science, Hebei Normal University, Shijiazhuang City, China 050024}
\address{The Fields Institute, 222 College Street, Toronto, Ontario, Canada, M5T, 3J1}
\email{sxdzq@yahoo.com}

\begin{document}
\newcommand{\li}{\lim\limits_{\longrightarrow}}

\title[KK--lifting problem for dimension drop interval
algebras]{KK--lifting problem for dimension drop interval algebras }

\begin{abstract}
In this paper, we investigate KK-theory of (generalized) dimension
drop interval algebras (with possibly different dimension drops at
the endpoints), especially on the problem that which KK-class is
representable by a $*$-homomorphism between two such C*-algebras
(allowing tensor product with matrix for the codomain algebra). This
lifting problem makes sense on its own in KK-theory , and also has
application on the classification of C*-algebras which are inductive
limits of these  building blocks. It turns out that when the
dimension drops at the two endpoints are different, there exist
KK-elements which preserve the order structure defined by M.
Dadarlat and T. A. Loring in \cite{DL1} on the mod $p$ K-theory, but
fails to be lifted to $*$-homomorphism. This is different from the
equal dimension drops case as shown by S. Eliers in \cite{Ei}.
\end{abstract}

\maketitle

\tableofcontents
\section{Introduction}
The Elliott program has been a very active research program in the
area of operator algebra. A great number of simple C*-algebras have
been classified by the standard Elliott invariant. For non-simple
case, however, as soon as there is torsion in $K_1$, even if the
algebra has real rank zero, the graded ordered K-group is no longer
sufficient. This was shown first by G. Gong in \cite{Go} for the
case of approximately homogeneous C*-algebras (AH-algebras), and was
shown by M. Dadarlat and T. A. Loring in \cite{DL1} for the case of
approximately (classical) dimension drop C*-algebras considered by
Elliott in \cite{Ell}. (In other words, while the main theorem of
\cite{Ell} was correct as it dealt with circles, it was not correct
in the setting of dimension drop algebras, except in the simple case
--- in which case, interestingly, the proof was also essentially
correct.)

M. Dadarlat and T. A. Loring did prove an isomorphism theorem for
real rank zero inductive limits of (classical) dimension drop
algebras, using the total K-theory group (i.e., K-theory with
coefficient) introduced by G. Gong in \cite{Go} (for the purpose of
giving a counterexample), together with a new order structure which
they defined in a completely general text, but they had to assume
that an order isomorphism arose from a KK-element. For the case of
bounded torsion in $K_1$-group, this assumption was avoided by S.
Eilers in \cite{Ei}; and in \cite{DL2}, M. Dadarlat and T. A. Loring
showed that the KK-element existed in very great generality (the
well--known universal multicoefficient theorem). This was extended
by M. Dadarlat and G. Gong to include a classification of the
AH-algebras considered by Gong in \cite{Go} (and by Elliott and Gong
in [EG] in the simple case).

All this refers only to the original (classical) dimension drop
interval algebras introduced by Elliott in \cite{Ell}. The
non-completeness of the order structure defined by Elliott on the
$K_{*}$-group is essentially related to the KK-lifting problem in
real rank zero setting. Namely, there exist KK-elements between the
classical dimension drop algebras which preserve this order
structure, but fail to be represented by a $*$-homomorphism between
two such algebras. The new order structure on K-theory with
coefficient defined by M. Dadarlat and T. A. Loring together with
the so called Bockstein operations (see also \cite{DL1}) can exactly
solve this KK-lifting question.

Now, we work on the generalized dimension drop algebras, of course
we need to investigate the KK-lifting problem between two such
algebras. The main purpose of the present paper is to show that
generalized dimension drop interval algebras, with different
dimension drops at the two ends of the interval, raise a new
problem. As we show, the existence theorem with the Dadarlat-Loring
order structure on total K-theory fails at the level of building
blocks.
\begin{thrm}\label{main} For
generalized dimension drop algebras $A_{m}=I[m_0,m,$ $m_1]$ and
$B_{n}=I[m_0,n,m_1]$ with $(m_0,m_1)=1$, there exists KK-elements in
$KK(A_m,B_n)$ (certain linear combinations of horizontal eigenvalue
patterns with non symmetric coefficients), such that they preserve
the Dadarlat-Loring order structure on total K-theory, but fail to
be lifted to a $*$-homomorphism between $A_m$ and $B_n$.
\end{thrm}
To achieve this, we realize each KK-element on both ordered K-groups
with coefficient and ordered K-homology groups of two such algebras.
On one hand, we need to calculate the K-groups with coefficient for
the generalized dimension drop algebras, the positive cone and all
the related Bockstein operations. Moreover, we need to have complete
analysis on the structure of induced elements of KK-elements on
ordered K-groups with coefficient, from this analysis we can obtain
the condition under which a KK-element can preserve the
Dadarlat-Loring order structure. On the other hand, X. Jiang and H.
Su studied these generalized dimension drop interval algebras in
\cite{JS}, they gave a criterion for KK-lifting problem between two
such algebras. Their criterion is some positivity on the K-homology
groups of these algebras. By our analysis, the different dimension
drops at the two endpoints give us flexibility for KK-elements to
preserve the Dadarlat-Loring order but do not satisfy Jiang-Su's
criterion for KK-lifting. Therefore, we succeed in finding the
KK-elements in the main theorem above.

The paper is organized as follows. In section 2, some preliminaries
are given about the generalized dimension drop algebras and the mod
$p$ K-theory, including the Dadarlat-Loring order structure and all
the Bockstein operations we need. In section 3, we calculate the
K-theory with coefficient for generalized dimension drop algebras,
including the positive cone, etc. In section 4, we make analysis on
the behavior of KK-elements on K-theory with coefficient, and figure
out the structure of morphisms on K-theory with coefficient. In
section 5, we investigate the KK-lifting problem for generalized
dimension drop algebras, and prove the main Theorem \ref{main},
namely, give examples of KK-elements which preserve the
Dadarlat-Loring order but fail be lifted to a $*$-homomorphism.

\section{Notation and preliminaries}

\begin{definition} A dimension drop algebra, denoted by
$\textbf{I}[m_0,m,m_1]$, is a C*-algebra of the
form:$$\textbf{I}[m_{0},m,m_{1}]=\{f\in
\textrm{C}([0,1],\textrm{M}_m):f(0)=a_{0}\otimes
\textrm{id}_{\frac{m}{m_{0}}},
f(1)=\textrm{id}_{\frac{m}{m_{1}}}\otimes a_{1}\},$$ \noindent where
$a_0$ and $a_1$ belong to $\textrm{M}_{m_0}(\mathbb{C})$ and
$\textrm{M}_{m_1}(\mathbb{C})$ respectively.
\end{definition}

Note that the classical dimension drop interval algebras, both the
non-unital one $\textrm{I}_p=\{f\in
\textrm{C}([0,1],\textrm{M}_p):f(0)=0, f(1)\in \mathbb{C}\}$ and the
unital one $\tilde{\textrm{I}}_p=\{f\in
\textrm{C}([0,1],\textrm{M}_p):f(0), f(1)\in \mathbb{C}\}$ are
included in the definition above. For the (generalized) dimension
drop interval algebra, the two singular irreducible representations
at the endpoints $V_0(f)=a_0$ and $V_1(f)=a_1$ are important for
further analysis, these two representations exactly reflect the
information of different dimension drops.

Moreover, there are four basic $*$-homomorphisms $\delta_0$,
$\delta_1$, $id_{m,n}$, and $\overline{id_{m,n}}$ between such
dimension drop interval algebras which play central roles later, we
introduce them here:
$$\delta_i: \textbf{I}[m_0,m,m_1]\longrightarrow
\textbf{I}[m_{0},n,m_1]\otimes \textrm{M}_{m_i}, i=0,1,$$ is defined
by
$$\delta_{i}(f)=\left(
                                                              \begin{array}{cccc}
                                                                V_{i}(f) &  &  &  \\
                                                                 & V_{i}(f) &  &  \\
                                                                 &  & \ddots &  \\
                                                                 &  &  & V_{i}(f) \\
                                                              \end{array}
                                                            \right).
$$ we need $(m_0,m_1)=1$ for this.
$$id_{m,n}: \textbf{I}[m_0,m,m_1]\longrightarrow
\textbf{I}[m_{0},n,m_1]\otimes \textrm{M}_{\frac{m}{(m,n)}}$$ is
defined by
$$id_{m,n}(f)(t)=\left(
          \begin{array}{cccc}
            f(t) &  &  &  \\
             & f(t) &  &  \\
             &  & \ddots &  \\
             &  &  & f(t) \\
          \end{array}
        \right),
$$ where $f(t)$ repeats $\frac{n}{(m,n)}$ times.
$$\overline{id}_{m,n}: \textbf{I}[m_0,m,m_1]\longrightarrow
\textbf{I}[m_{0},n,m_1]\otimes \textrm{M}_{\frac{m}{((m,n),s)}},$$
is defined by $$\overline{id}_{m,n}(f)(t)=\left(
                                      \begin{array}{cccc}
                                        f(1-t) &  &  &  \\
                                         & f(1-t) &  &  \\
                                         &  & \ddots &  \\
                                         &  &  & f(1-t)\\
                                      \end{array}
                                    \right)
,$$ where $s=\frac{m}{m_0m_1}$, we also need $(m_0,m_1)=1$. Note
that the non symmetric sizes appear above are also caused by the
different dimension drops at the two endpoints.

In the next, we give a few basic preliminaries about K-theory with
coefficient.

Given a natural number $p\geq 2$ (not necessarily a prime), denote
$\mathbb{Z}/p\mathbb{Z}$ by $\mathbb{Z}_p$, the mod $p$ K-theory for
C*-algebras was studied by \cite{Cu}, \cite{S0} and \cite{S}. Given
any C*-algebra $A$, to define $K_0(A;\mathbb{Z}_p)$, one may choose
any nuclear C*-algebra $P$ in the bootstrap class of \cite{RS} such
that $K_0(P)=0$ and
 $K_1(P)=\mathbb{Z}_p$, then $K_0(A;\mathbb{Z}_p)\cong K_1(A\otimes P)\cong KK(P,A)$ (see \cite{S}, \cite{Bla}). Conventionally, we choose $P$ to be the
 classical dimension drop interval algebra $\textrm{I}_p$. Throughout this paper, let
 $G_p=\mathbb{Z}\oplus\mathbb{Z}_p$, then $K_0(A;G_p)$ is defined as
 $K_0(A)\oplus K_0(A;\mathbb{Z}_p)$, and we use $\bar{n}$ to denote the mod $p$ congruence class of the natural number $n$. By using the split short exact
 sequence $0\rightarrow \textrm{I}_p\rightarrow \tilde{\textrm{I}}_p\rightarrow \mathbb{C}\rightarrow 0$,
 one can see that $K_0(A;G_{p})=KK(\tilde{\textrm{I}}_{p},A)$.
So we
 summarize this as definition.
 \begin{definition} For any C*-algebra $A$, and any natural number
 $p\geq 2$, the mod $p$ K-theory is defined as follows:

 (i) $K_{0}(A; \mathbb{Z}_{p})\triangleq KK(\textrm{I}_{p}, A)$,

 (ii) $K_0(A; G_p)\triangleq KK(\tilde{\textrm{I}}_p, A)$.

 Similarly, the mod $p$ $K_1$-group can also be defined.
\end{definition}

As we mentioned in the introduction, the crucial thing about mod $p$
K-theory is the order structure defined by M. Dadarlat and T. A.
Loring on $K_0(A; G_p)$, and together with the Bockstein operations
(see \cite{DL1}).

\begin{definition} (Dadarlat-Loring order) $$K^{+}_{0}(A;G_p)\triangleq
\{([\varphi(1)], [\varphi|_{I_p}])\,|\, \varphi \in Hom(\widetilde{
I_p}, M_{k}(A))\, \text{for some integer}\, k \}.$$
\end{definition}

\begin{lem} (see \cite{DL1}) There is a natural short exact sequence of
groups:$$\textrm{K}_0(\textrm{A})\stackrel{\times p}\rightarrow
\textrm{K}_0(\textrm{A})\stackrel{\mu_{\textrm{A};p}}\rightarrow
\textrm{K}_0(\textrm{A};\mathbb{Z}_p)\stackrel{\nu_{\textrm{A};p}}\rightarrow
\textrm{K}_1(\textrm{A})\stackrel{\times p}\rightarrow
\textrm{K}_1(\textrm{A}).$$ where $p\geq 2$, $\mu_{\textrm{A};p},
\nu_{\textrm{A};p}$ are the Bockstein operations defined by the
Kasparov product with the element of
$\textrm{K}\!\textrm{K}(I_p,\mathbb{C})$ given by the evaluation
$\delta_1:\textrm{I}_p\rightarrow \mathbb{C}$ and the element of
$\textrm{K}\!\textrm{K}^1(\mathbb{C},I_p)$ given by the inclusion
$i:\textrm{S}\textrm{M}_p\rightarrow \textrm{I}_p$ respectively.
\end{lem}

\begin{lem}\label{comp} For any KK-element $\alpha\in
\textrm{K}\!\textrm{K}(\textrm{A},\textrm{B})$, where
$\textrm{A},\textrm{B}$ are two C*-algebras, then $\alpha$ induces
the following commutative
diagram:$$
            \begin{array}{ccccc}
              \textrm{K}_0(\textrm{A})\stackrel{\times p}\longrightarrow & \textrm{K}_0(\textrm{A})\stackrel{\mu_{\textrm{A};p}}\longrightarrow& \textrm{K}_0(\textrm{A};\mathbb{Z}_p)\stackrel{\nu_{\textrm{A};p}}\longrightarrow & \textrm{K}_1(\textrm{A})\stackrel{\times p}\longrightarrow & \textrm{K}_1(\textrm{A}) \\
               & \Big\downarrow\!{\textrm{K}_0(\alpha)} & \Big\downarrow\!{\textrm{K}_0(\alpha;\mathbb{Z}_p)} & \Big\downarrow\!{\textrm{K}_1(\alpha)} &  \\
              \textrm{K}_0(\textrm{B})\stackrel{\times p}\longrightarrow & \textrm{K}_0(\textrm{B})\stackrel{\mu_{\textrm{B};p}}\longrightarrow& \textrm{K}_0(\textrm{B};\mathbb{Z}_p)\stackrel{\nu_{\textrm{B};p}}\longrightarrow&  \textrm{K}_1(\textrm{B})\stackrel{\times p}\longrightarrow  & \textrm{K}_1(\textrm{B})\\
            \end{array}.
$$
\end{lem}

\begin{proof} This follows from the associativity of the Kasparov
product.
\end{proof}
 We will denote by
$\textbf{Hom}(\textbf{K}(\textrm{A};p),\textbf{K}(\textrm{B};p)$ the
group of the triples $(x,\varphi,y)$ such that the diagram above
commutes. For each $\alpha\in KK(A,B)$, the induced triple
$$\Gamma(\alpha;p)=(\textrm{K}_0(\alpha),\textrm{K}_0(\alpha;\mathbb{Z}_p),\textrm{K}_1(\alpha))$$ lives in
$\textbf{Hom}(\textbf{K}(\textrm{A};p),\textbf{K}(\textrm{B};p)$.

\section{K-theory with coefficient for generalized dimension drop interval algebras}
In previous section, we have four $*$-homomorphism between two
generalized dimension drop interval algebras. We also have these
homomorphisms from $\tilde{I}_p$ to $A_m=\textrm{I}[m_0,m,m_1]$ as
follows: $\delta_0, \delta_1, id$, and $\overline{id}$.
$$\delta_i: \tilde{I}_p\longrightarrow
\textbf{I}[m_{0},m,m_1], i=0,1,$$ is defined by
$$\delta_{i}(f)=\left(
                                                              \begin{array}{cccc}
                                                                V_{i}(f) &  &  &  \\
                                                                 & V_{i}(f) &  &  \\
                                                                 &  & \ddots &  \\
                                                                 &  &  & V_{i}(f) \\
                                                              \end{array}
                                                            \right).
$$ $V_i(f)$ repeats $m$ times.
$$id: \tilde{I}_p\longrightarrow
\textbf{I}[m_{0},m,m_1]\otimes \textrm{M}_{\frac{p}{(m,p)}}$$ is
defined by
$$id(f)(t)=\left(
          \begin{array}{cccc}
            f(t) &  &  &  \\
             & f(t) &  &  \\
             &  & \ddots &  \\
             &  &  & f(t) \\
          \end{array}
        \right),
$$ where $f(t)$ repeats $\frac{m}{(m,p)}$ times.
$$\overline{id}: \tilde{I}_p\longrightarrow
\textbf{I}[m_{0},m,m_1]\otimes \textrm{M}_{\frac{p}{(m,p)}},$$ is
defined by $$\overline{id}(f)(t)=\left(
                                      \begin{array}{cccc}
                                        f(1-t) &  &  &  \\
                                         & f(1-t) &  &  \\
                                         &  & \ddots &  \\
                                         &  &  & f(1-t)\\
                                      \end{array}
                                    \right)
,$$ where $f(1-t)$ also repeats $\frac{m}{(m,p)}$ times.

Before the calculation of the K-theory with coefficient for
generalized dimension drop interval algebras
$A_m=\textrm{I}[m_0,m,m_1]$, we need the following lemma in
\cite{DL1}.

\begin{lem}\label{basic}
For the complex number $\mathbb{C}$, we have that
$K_0(\mathbb{C};G_p)=\mathbb{Z}\oplus\mathbb{Z}_p$, and
$K^{+}_{0}(\mathbb{C};G_p)=\{(a,\bar{b})\in
\mathbb{Z}\oplus\mathbb{Z}_p\,|\, a\geq b\}$. Moreover, we have the
identification given by $[\delta_0]=(1,\bar{0})$, and
$[\delta_1]=(1,\bar{1})$.

\end{lem}

\begin{remark} Note that we not only use this result, but also the way of
identification, namely, for every $*$-homomorphism
$\phi:\tilde{\textrm{I}}_p\rightarrow \textrm{M}_n(\mathbb{C})$, by
using its irreducible representation, we assume $\phi$ can be
decomposed into $c_0$ copies of $\delta_0$, $c_1$ copies of
$\delta_1$, and $m$ copies of $\delta_t$ for $t\in (0,1)$, then
$[\phi]=(c_0+c_1+mp,\bar{c}_1)$.
\end{remark}

\begin{thrm}\label{calcu} For any natural number $p\geq2$, we have
that $\textrm{K}_0(A_m;\textrm{G}_p)$ $=\{(b',\bar{b},c',\bar{c})\in
(\mathbb{Z}\oplus \mathbb{Z}_p)\oplus(\mathbb{Z}\oplus
\mathbb{Z}_p)\,|\,\dfrac{m}{m_1}c'-\dfrac{m}{m_0}b'=0,\dfrac{m}{m_1}c-\dfrac{m}{m_0}b\in
p\mathbb{Z}\}\cong\mathbb{Z}\oplus \textrm{Z}(m,p)$, where
$\textrm{Z}(m,p)=\{(\bar{b},\bar{c})\in\mathbb{Z}_p\oplus\mathbb{Z}_p\,|\,\dfrac{m}{m_1}c-\dfrac{m}{m_0}b\in
p\mathbb{Z}\}$, and the isomorphism is given by
$(a,\bar{b},\bar{c})\in\mathbb{Z}\oplus
\textrm{Z}(m,p)\longrightarrow
(a\dfrac{m_0}{(m_0,m_1)},\bar{b},a\dfrac{m_1}{(m_0,m_1)},\bar{c})$.

$K^{+}_{0}(A_m;G_p)=\{(a,\bar{b},\bar{c})\in
\mathbb{Z}\oplus\mathbb{Z}(m,p)\,|\,a\dfrac{m_0}{(m_0,m_1)}\geq b,
a\dfrac{m_1}{(m_0,m_1)}\geq c\}$.

And, we have the following identifications:
\begin{align*}
&[\delta_0]=((m_0,m_1),\bar{0},\bar{0}),\;[\delta_1]=((m_0,m_1),\bar{m}_0,\bar{m}_1),
\\
&[id]=\dfrac{p}{(p,m)}((m_0,m_1),\bar{0},\bar{m}_1),
\;[\overline{id}]=\dfrac{p}{(p,m)}((m_0,m_1),\bar{m}_0,\bar{0}).
\end{align*} If we further assume that $(m_0,m_1)=1$ and $m|p$, then
$[\delta_0], [\delta_1]$, and $[id]$ generate $K_0(A_m; G_p)$. Then,
the Bockstein operations are given by

$\mu_{A_{m};\,p}=\left(
                   \begin{array}{c}
                     m_0 \\
                     m_1 \\
                   \end{array}
                 \right)
$, and $\nu_{A_{m};\,p}=(-\dfrac{m}{pm_0}, \dfrac{m}{pm_1})$.
\end{thrm}
\begin{proof} Consider the following short exact sequence $$0\rightarrow \textrm{SM}_m\rightarrow \textrm{I}[m_0,m,m_1]\rightarrow \textrm{M}_{m_0}\oplus \textrm{M}_{m_1}\rightarrow
0,\qquad{(*)}$$ it induces the following six term exact sequence for
KK-groups: $$
\xymatrix{ KK(\tilde{I}_{p},\textrm{SM}_m)  \ar[r]  & KK(\tilde{I}_{p}, \textrm{I}[m_{0},m,m_{1}])\ar[r] & KK(\tilde{I}_{p}, \textrm{M}_{m_0}\oplus \textrm{M}_{m_1}) \ar[d]^{\partial}\\
KK^{1}(\tilde{I}_{p}, \textrm{M}_{m_0}\oplus \textrm{M}_{m_1})
\ar[u]^{\partial}  & KK^{1}(\tilde{I}_{p},
\textrm{I}[m_{0},m,m_{1}]) \ar[l] &
KK^{1}(\tilde{I}_{p},\textrm{SM}_m). \ar[l]} $$ We know that
$KK(\tilde{I}_P, \textrm{SM}_m)=K^1(\tilde{I}_p)=0$, see for example
(\cite{JS}, Lemma 3.1), and $$KK(\tilde{I}_p, \textrm{M}_{m_0}\oplus
\textrm{M}_{m_1})=(\mathbb{Z}\oplus\mathbb{Z}_p)\oplus(\mathbb{Z}\oplus\mathbb{Z}_p),
KK^1(\tilde{I}_p, \textrm{SM}_m)=\mathbb{Z}\oplus\mathbb{Z}_p$$ by
Lemma \ref{basic}. Hence by exactness, we obtain that
$$KK(\tilde{I}_{p}, \textrm{I}[m_{0},m,m_{1}])=\textrm{Ker} \partial,\, KK^1(\tilde{I}_{p}, \textrm{I}[m_{0},m,m_{1}])=(\mathbb{Z}\oplus\mathbb{Z}_p)/\textrm{Im}\partial.$$
To calculate $KK(\tilde{I}_{p}, \textrm{I}[m_{0},m,m_{1}])$, it is
enough to figure out the index map $\partial$ on the right hand of
the diagram above.

To do this, we use the identification in Lemma \ref{basic}, assume
$(b', \bar{b}, c', \bar{c})\in (\mathbb{Z}\oplus \mathbb{Z}_p)\oplus
(\mathbb{Z}\oplus(\mathbb{Z}_p))$, since the index map $\partial$ is
induced by the extension $(*)$, one gets that
$$\partial(b',\bar{b},c',\bar{c})=(\dfrac{m}{m_1}c'-\dfrac{m}{m_0}b',\overline{\dfrac{m}{m_1}c-\dfrac{m}{m_0}b}).$$
Hence, $$KK(\tilde{I}_{p},
\textrm{I}[m_{0},m,m_{1}])=\{(b',\bar{b},c',\bar{c})\,|\,\dfrac{m}{m_1}c'-\dfrac{m}{m_0}b'=0,
\dfrac{m}{m_1}c-\dfrac{m}{m_0}b\in p\mathbb{Z}\}.$$ Similarly,
$KK^1(\tilde{I}_p, \textrm{I}[m_0,m,m_1])$ can also be calculated,
but we don't need it here.

Let
$\textrm{H}=\{(b',c')\in\mathbb{Z}\oplus\mathbb{Z}\,|\,\frac{m}{m_1}c'-\frac{m}{m_0}b'=0\}$,
then $\textrm{H}\cong \mathbb{Z}$ via the map $(b',c')\rightarrow
\dfrac{c'+b'}{m_0+m_1}(m_0,m_1)$, so
$(\dfrac{m_0}{(m_0,m_1)},\dfrac{m_1}{(m_0,m_1)})$ corresponds to
$1\in \mathbb{Z}$, then the inverse of this map is
$a\in\mathbb{Z}\rightarrow (a\dfrac{m_0}{(m_0,m_1)},
a\dfrac{m_1}{(m_0,m_1)})$. Therefore, by the identification od Lemma
\ref{basic}, we obtain the positive cone of $K_0(A_m; G_p)$:
$$K^{+}_{0}(A_m;G_p)=\{(a,\bar{b},\bar{c})\in
\mathbb{Z}\oplus\mathbb{Z}(m,p)\,|\,a\dfrac{m_0}{(m_0,m_1)}\geq b,
a\dfrac{m_1}{(m_0,m_1)}\geq c\}.$$

To see the corresponding elements for the homomorphisms $\delta_0$,
$\delta_1$, $id$, and $\overline{id}$, we just need to go through
the identification of Lemma \ref{basic}. This is straightforward,
but easily confused, so we spell out these calculations here for the
convenience of readers.

For $\delta_0: \tilde{I}_p\rightarrow \textrm{I}[m_0,m,m_1]$, apply
the quotient map $\pi: \textrm{I}[m_0,m,m_1]\rightarrow
\textrm{M}_{m_0}\oplus \textrm{M}_{m_1}$, we get
$$\pi(\delta_0(f))=\left(
                                                                   \begin{array}{ccc}
                                                                     V_0(f) &  &  \\
                                                                      & \ddots &  \\
                                                                      &  & V_0(f) \\
                                                                   \end{array}
                                                                 \right)\oplus \left(
                                                                   \begin{array}{ccc}
                                                                     V_0(f) &  &  \\
                                                                      & \ddots &  \\
                                                                      &  & V_0(f) \\
                                                                   \end{array}
                                                                 \right),
$$ $V_0(f)$ repeats $m_0$ times in the first direct summand, and
repeats $m_1$ times in the second direct summand. So by the
identification in Lemma \ref{basic}, this gives us the group element
$(m_0, \bar{0}, m_1, \bar{0})\in
(\mathbb{Z}\oplus\mathbb{Z}_p)\oplus
(\mathbb{Z}\oplus\mathbb{Z}_p)$; Similarly, $\delta_1$ gives us the
group element $(m_0, \bar{m}_0, m_1, \bar{m}_1)\in
(\mathbb{Z}\oplus\mathbb{Z}_p)\oplus
(\mathbb{Z}\oplus\mathbb{Z}_p)$. For the map $id:
\tilde{I}_p\rightarrow
\textrm{M}_{\frac{p}{(p,m)}}(\textrm{I}[m_0,m,m_1])$, apply the
quotient map $\pi$, we get $$\pi(id(f))=\left(
                                                                   \begin{array}{ccc}
                                                                     V_0(f) &  &  \\
                                                                      & \ddots &  \\
                                                                      &  & V_0(f) \\
                                                                   \end{array}
                                                                 \right)\oplus \left(
                                                                   \begin{array}{ccc}
                                                                     V_1(f) &  &  \\
                                                                      & \ddots &  \\
                                                                      &  & V_1(f) \\
                                                                   \end{array}
                                                                 \right),$$ $V_0(f)$ repeats $\dfrac{pm_0}{(p,m)}$ times in the first direct summand, and
$V_1(f)$ repeats $\dfrac{pm_1}{(p,m)}$ times in the second direct
summand. Hence, it gives the element $\frac{p}{(p,m)}(m_0, \bar{0},
m_1, \bar{m}_1)\in (\mathbb{Z}\oplus\mathbb{Z}_p)\oplus
(\mathbb{Z}\oplus\mathbb{Z}_p)$. Similarly, the map $\overline{id}$
corresponds the element $\frac{p}{(p,m)}(m_0, \bar{m}_0, m_1,
\bar{0})\in (\mathbb{Z}\oplus\mathbb{Z}_p)\oplus
(\mathbb{Z}\oplus\mathbb{Z}_p)$.

If we further assume that $(m_0,m_1)=1$, and $m|p$, then
$$[\delta_0]=(1,\bar{0},\bar{0}),[\delta_1]=(1,\bar{m}_{0},\bar{m_1}),[id]=\frac{p}{m}(1,\bar{0},\bar{m_1}),
[\overline{id}]=\frac{p}{m}(1,\bar{m_0},\bar{0}).$$ Moreover,
$(\bar{m}_0,\bar{m}_1)$ is the smallest choice of non-negative pairs
$(b,c)$ such that $\dfrac{m}{m_1}c-\dfrac{m}{m_0}b=0$ (we have
$(m_0,m_1)=1$); by the assumption $m|p$, we know that
$(0,\dfrac{p}{m}\bar{m}_1)$ is the smallest choice of non-negative
integer pairs whose first coordinate is zero and such that
$\dfrac{m}{m_1}c-\dfrac{m}{m_0}b=p$. Then the linear combination of
$(\bar{m}_{0},\bar{m}_{1})$, and $(0,\dfrac{p}{m}\bar{m}_{1})$ with
integer coefficients (could be negative numbers) will be the group
$K_{0}(A_m, \mathbb{Z}/p\mathbb{Z})$. Therefore, for any element
$\gamma=(x,\bar{b},\bar{c})\in K_0(A_m;G_p)$, there exist integers
$l_1$ and $l_2$ (could be negative numbers), such that
$\gamma=(x,l_{1}\bar{m}_0, l_{2}\bar{m}_1)$. Suppose
$$\gamma=c_{1}[\delta_0]+c_{2}[\delta_1]+c_{3}[id], \qquad{(*)}$$ we can always
set $c_2=l_1$, $l_2-l_1=\frac{p}{m}c_3$, since $\gamma\in K_0(A_m;
G_p)$, we have $\dfrac{m}{m_1}l_{2}m_1-\dfrac{m}{m_0}l_{1}m_0=pj$
for some integer $j$. Therefore, $c_3=j$, then $c_1=x-l_2$. So we
can find integers $c_1,c_2,c_3$ which satisfy $(*)$.

In the next, we calculate the Bockstein operations $\mu_{A_m;p}$ and
$\nu_{A_m;p}$. To do so, we need to go through the definition of
these operations on the generators, actual compositions of maps on
generators. Recalling that the generators of $K_0(A_m)$ is the
matrix value projection $h(t)$ such that $h(0)=a_0\otimes
id_{\frac{m}{m_0}}$, $a_0=1\otimes id_{m_0}$; and $h(1)=a_1\otimes
id_{\frac{m}{m_1}}$, $a_1=1\otimes id_{m_1}$. Then the composition
with the map $\delta_1: I_p\rightarrow \mathbb{C}$ gives us the map
$\delta_1(f)h(t)$ from $I_p$ to $A_m$. By our identification, this
map corresponds the element $(\bar{m}_0,\bar{m}_1)\in Z(m,p)$. So
$\mu_{A_m;p}=\left(
               \begin{array}{c}
                 m_0 \\
                 m_1 \\
               \end{array}
             \right).
$

For the Bockstein operation $\nu_{A_m;p}$, note that the generator
of $K_1(A_m)$
$(K_1(I_p))$ is the matrix value function $$g(t)=\left(
                                                               \begin{array}{cccc}
                                                                 e^{2\pi it} &  &  &  \\
                                                                  & 1 &  &  \\
                                                                 &  & \ddots &  \\
                                                                  &  &  & 1 \\
                                                               \end{array}
                                                             \right),
$$ then the composition of the map $1\rightarrow g(t)$ with $id$
gives the element $\overline{\frac{m}{(p,m)}}\in K_1(A_m)$, and with
$\delta_i, i=0,1$ gives zero (then the composition with
$\overline{id}$ gives $-\overline{\frac{m}{(p,m)}}$). Hence
$\nu_{A_m;p}=(-\dfrac{m}{pm_0}, \dfrac{m}{pm_1})$.

\end{proof}

Although $[\delta_0],[\delta_1]$ and $[id]$ can already generate the
group $K_{0}(A_m; G_p)$, but it is not true that the positive cone
is the linear span of these three elements with non-negative integer
coefficients, for example, the element $[\overline{id}]$ can not be
written as a linear combination of these three with non-negative
integer coefficients. For our purpose, the following stronger
version of generators of the positive cone is needed.

\begin{lem}\label{lem:generator of cone} Given a generalized dimension drop algebra
$A_m=\textrm{I}[m_0,m,$ $m_1]$ with $(m_0,m_1)=1$, and any positive
integer $p$ with $m$ divides $p$, then any element of the
Dadarlat-Loring positive cone of $K_{0}(A_m; G_{p})$ can be written
as a linear combination of $[\delta_0],[\delta_1],[id]$, and
$[\overline{id}]$ with non-negative integer coefficients.
\end{lem}

\begin{proof} Recalling the calculation of the K-theory with
coefficient, we know that: $$K_{0}(A_m,
\mathbb{Z}\oplus\mathbb{Z}/p\mathbb{Z})=\mathbb{Z}\oplus\{(\bar{b},\bar{c})\in
\mathbb{Z}/p\mathbb{Z}\oplus
\mathbb{Z}/p\mathbb{Z}\,|\,\frac{m}{m_1}c-\frac{m}{m_0}b\in
p\mathbb{Z} \},$$ and
$$[\delta_0]=(1,\bar{0},\bar{0}),[\delta_1]=(1,\bar{m}_{0},\bar{m_1}),[id]=\frac{p}{m}(1,\bar{0},\bar{m_1}), [\overline{id}]=\frac{p}{m}(1,\bar{m_0},\bar{0}).$$
Since $(\bar{m}_0,\bar{m}_1)$ is the smallest choice of non-negative
pairs $(b,c)$ such that $\dfrac{m}{m_1}c-\dfrac{m}{m_0}b=0$ (we have
$(m_0,m_1)=1$); by the assumption $m|p$, we know that
$(0,\dfrac{p}{m}\bar{m}_1)$ is the smallest choice of non-negative
integer pairs whose first coordinate is zero and such that
$\dfrac{m}{m_1}c-\dfrac{m}{m_0}b=p$, similarly,
$(\dfrac{p}{m}\bar{m}_0,0)$ is the smallest choice of non-negative
integer pairs whose second coordinate is zero and such that
$\dfrac{m}{m_1}c-\dfrac{m}{m_0}b=p$. Then the linear combination of
$(\bar{m}_{0},\bar{m}_{1}), (\dfrac{p}{m}\bar{m}_{0},0)$, and
$(0,\dfrac{p}{m}\bar{m}_{1})$ with non-negative integer coefficients
will be the group $K_{0}(A_m, \mathbb{Z}/p\mathbb{Z})$.

Hence, for any element $\gamma=(x,\bar{b},\bar{c})$ in the positive
cone of $K_{0}(A_m, \mathbb{Z}\oplus\mathbb{Z}/p\mathbb{Z})$, there
exist non-negative integers $l_1$ and $l_2$ such that $\gamma=(x,
l_{1}\bar{m}_{0},$
$ l_{2}\bar{m}_{1})$. In the next, we try to write
$\gamma$ as the linear combination of $[\delta_0],[\delta_1],[id]$,
and $[\overline{id}]$ with positive integer coefficients.

 Suppose
$$\gamma=c_{1}[\delta_0]+c_{2}[\delta_1]+c_{3}[id]+c_{4}[\overline{id}],$$
then one obtains that
\begin{align*}
&x=c_{1}+c_{2}+c_{3}\frac{p}{m}+c_{4}\frac{p}{m},\\
&l_{1}=c_{2}+c_{4}\frac{p}{m}, \hspace{2.0cm}\qquad(**)\\
&l_{2}=c_{2}+c_{3}\frac{p}{m}.
\end{align*}

In case 1: assume that $l_{1}\geq l_{2}$, then
$\dfrac{p}{m}c_{4}=\dfrac{p}{m}c_{3}+l_{1}-l_{2}$, let $c_{3}=0$,
one gets $\dfrac{p}{m}c_{4}=l_{1}-l_{2}$. Since
$\dfrac{m}{m_1}l_{2}m_{1}-\dfrac{m}{m_0}l_{1}m_{0}\in p\mathbb{Z}$,
i.e., $l_{2}-l_{1}=\dfrac{p}{m}j$ for some integer $j\leq 0$. Then
$c_{4}=-j\geq 0$, therefore, $c_{2}=l_{2}=l_{1}+j\dfrac{p}{m}$. This
gives us
$c_{1}=x-l_{2}-\dfrac{p}{m}(-j)=x-l_{2}+l_{2}-l_{1}=x-l_{1}$.
Because $\gamma$ is in the positive cone, we know that $x\geq
l_{1}$, namely, $c_1$ is non-negative.

In case 2: assume that $l_{2}\geq l_{1}$, then
$\dfrac{p}{m}c_{3}=\frac{p}{m}c_{4}+l_{2}-l_{1}$, let $c_{4}=0$, one
gets that $\dfrac{p}{m}c_{3}=l_{2}-l_{1}$. Since
$\dfrac{m}{m_1}l_{2}m_{1}-\dfrac{m}{m_0}l_{1}m_{0}\in p\mathbb{Z}$,
namely, $l_{2}-l_{1}=\dfrac{p}{m}j$ for some integer $j\geq 0$. Then
$c_{3}=j\geq 0$, therefore, $c_{2}=l_{1}=l_{2}-\dfrac{p}{m}j$, this
gives us $c_{1}=x-l_{1}-\dfrac{p}{m}j=x-l_{2}$. Because $\gamma$ is
in the positive cone, we know that $x\geq l_{2}$, namely, $c_1$ is
non-negative.

In all cases, we can always find non-negative integer solution for
the system $(**)$ above. Then we are done.
\end{proof}

\section{Morphisms between K-theory with coefficient}

In this section, we will look at the morphisms between K-theory with
coefficient of two generalized dimension drop interval algebras,
also preserving the Bockstein operations. Given
$A_m=\textrm{I}[m_0,m,m_1]$, and $B_n=\textrm{I}[m_0,n,m_1]$, recall
Lemma \ref{comp} and the notation there, we investigate the
structure of
$\textbf{Hom}(\textbf{K}(\textrm{A}_m;p),\textbf{K}(\textrm{B}_n;p)$.
For each KK-element $\alpha\in KK(A_m,B_n)$, the induced triple
$$\Gamma(\alpha;p)=(\textrm{K}_0(\alpha),\textrm{K}_0(\alpha;\mathbb{Z}_p),\textrm{K}_1(\alpha))$$
lives in
$\textbf{Hom}(\textbf{K}(\textrm{A}_m;p),\textbf{K}(\textrm{B}_n;p)$.
Therefore, to analyze the the structure of
$\textbf{Hom}(\textbf{K}(\textrm{A}_m;p),\textbf{K}(\textrm{B}_n;p)$
is crucial in the following two senses: first, it is useful to
figure out the condition under which a KK-element can preserve the
Dadarlat-Loring order structure; second, more importantly, this
analysis indicate the possible candidates of KK-elements which
preserve the Dadarlat-Loring order structure, but fail to be
representable by $*$-homomorphisms.

To analyze the structure of
$\textbf{Hom}(\textbf{K}(\textrm{A}_m;p),\textbf{K}(\textrm{B}_n;p)$
means to give a general description of how an element looks like. A
direct guess one could make is by direct calculation for a general
element here. However, the equations set up from Lemma \ref{comp}
for a general element in
$\textbf{Hom}(\textbf{K}(\textrm{A}_m;p),\textbf{K}(\textrm{B}_n;p)$
is quite complicate, since we are working in mod $p$ setting, to
solve such equations is also difficult. Instead, we first give
description of special elements of the form $\Gamma=(0,\varphi,0)$;
and then construct concrete elements $(x,\phi,0)$ and $(0,\psi,y)$
for given $K_0$-multiplicity $x$ and $K_1$-multiplicity $y$. Then we
can have a description for general elements in
$\textbf{Hom}(\textbf{K}(\textrm{A}_m;p),\textbf{K}(\textrm{B}_n;p)$.
All these arguments become much easier for the classical dimension
drop interval algebras due to the fact $m_0=m_1=1$ (equivalently the
equal dimension drop).

\begin{lem}\label{torsion} Given $A_m=I[m_0,m,m_1]$ and $B_n=I[m_0,n,m_1]$, and any positive integer $p$ with $m|p$. For any induced triple $\gamma=(x, \varphi, y)$, if $\gamma=(0, \varphi,
0)$, then $\varphi=d\left(
                               \begin{array}{cc}
                                 -m_1m_0 & m_0m_0 \\
                                 -m_1m_1 & m_0m_1 \\
                               \end{array}
                             \right)
$ for some integer d with $0\leq d<\dfrac{m}{m_0m_1}$.
\end{lem}
\begin{proof} Recalling from Lemma \ref{comp}, $\gamma$ should fit
the following commutative
diagram:$$\xymatrix{K_{0}(A_m)\ar[r]^{\mu_{A_m;p}} \ar[d]^{\times x} &K_{0}(A_m,\mathbb{Z}_p) \ar[r]^{\nu_{A_m;p}} \ar[d]^{\varphi} &K_{1}(A_m)\ar[d]^{\times y}\\
K_{0}(B_n)\ar[r]^{\mu_{B_n;p}}&K_{0}(B_n,\mathbb{Z}_p)\ar[r]^{\nu_{B_n;p}}&K_1(B_n)
.}$$ Suppose $\gamma=(0, \varphi, 0)$, then from the first (left)
commuting square above, we get $$\varphi\circ \left(
                                               \begin{array}{c}
                                                 m_0 \\
                                                 m_1 \\
                                               \end{array}
                                             \right)(1)=0,
                                             \hspace{2.0cm}\qquad{(1)}
$$ and from the second (right) commuting square, we get
$$\nu_{B_n;p}\circ\varphi\left(
                          \begin{array}{c}
                            \bar{b} \\
                            \bar{c} \\
                          \end{array}
                        \right)=0 \hspace{2.0cm}\qquad{(2)}
$$ for any $(\bar{b},\bar{c})\in Z(m,p)$.

Assume $\varphi=\left(
                  \begin{array}{cc}
                    \varphi_{11} & \varphi_{12} \\
                    \varphi_{21} & \varphi_{22} \\
                  \end{array}
                \right)
$, from (1), we get $\varphi\left(
                           \begin{array}{c}
                             \bar{m}_0 \\
                             \bar{m}_1 \\
                           \end{array}
                         \right)=0,
$ namely, \begin{align} &\varphi_{11}m_0+\varphi_{12}m_1=pj\\
&\varphi_{21}m_0+\varphi_{22}m_1=pi.
\end{align} Apply (2) on the element $(\bar{0},\dfrac{p}{m}\bar{m_1})$, we obtain
$$(-\dfrac{n}{pm_0}, \dfrac{n}{pm_1})\left(
                                      \begin{array}{c}
                                        \varphi_{12}\dfrac{p}{m}\bar{m}_1 \\
                                        \varphi_{22}\dfrac{p}{m}\bar{m}_1 \\
                                      \end{array}
                                    \right)=\varphi_{22}\dfrac{n}{m}-\varphi_{12}\dfrac{nm_1}{mm_0}\in p\mathbb{Z}.
$$ Since $(m_0,m_1)=1$, there exist $\beta_0, \beta_1\in\mathbb{Z}$,
such that $\beta_0m_0+\beta_1m_1=1$. So we have
$pj\beta_0m_0+pj\beta_1m_1=pj$, this equation subtract (4.1), we get
$$(\varphi_{11}-pj\beta_0)m_0+(\varphi_{12}-pj\beta_1)m_1)=0. \hspace{2.0cm} (3)$$ Since $(m_0,m_1)=1$, one has
that $\varphi_{11}-pj\beta_0=d_1m_1$, and
$pj\beta_1-\varphi_{12}=d_1m_0$ for some integer $d_1$. Because
$\varphi$ is a morphism in the mod $p$ setting, so we have
$\varphi_{11}=-d_1m_1, \varphi_{12}=d_1m_0$. Similarly, the the same
argument applies on (4.2), we have that $\varphi_{21}=-d_2m_1,
\varphi_{22}=d_2m_0$. Combine with (3), one gets
$\dfrac{d_2nm_0}{m}-\dfrac{d_1nm_1m_0}{mm_0}=pk$ for some integer
$k$. Then $d_2nm_0-d_1nm_1=pkm$, use the same argument above with
the pair $(\beta_0,\beta_1)$, we obtain that
$$(d_2n-\beta_0pkm)m_0-(d_1n+\beta_1pkm)m_1=0.$$ Still note that
$\varphi$ is a morphism in the mod $p$ setting, we have that
$d_2n=rm_1$ and $d_1n=rm_0$ for some integer $r$, from here, we get
$d_2m_0=d_1m_1$. Therefore $d_2=dm_1$ and $d_1=dm_0$ for some
integer $d$. Hence, $\varphi=d\left(
                               \begin{array}{cc}
                                 -m_1m_0 & m_0m_0 \\
                                 -m_1m_1 & m_0m_1 \\
                               \end{array}
                             \right).
$

Suppose that $\varphi=d\left(
                               \begin{array}{cc}
                                 -m_1m_0 & m_0m_0 \\
                                 -m_1m_1 & m_0m_1 \\
                               \end{array}
                             \right)=0$, then this happens if and
                             only if $\varphi\left(
                                               \begin{array}{c}
                                                 \bar{m}_0 \\
                                                 \bar{m}_1 \\
                                               \end{array}
                                             \right)=0,
                             $ and $\varphi\left(
                                               \begin{array}{c}
                                                 \bar{0} \\
                                                 \dfrac{p}{m}\bar{m}_1 \\
                                               \end{array}
                                             \right)=0$. The first
                                             equation is
                                             automatically satisfied
                                             since $\varphi$ has
                                             this special form. For
                                             the second one, it
                                             means that
                                             $\dfrac{dm_0m_0pm_1}{m}\in p\mathbb{Z}$
                                             and
                                             $\dfrac{dm_0m_1pm_1}{m}\in
                                             p\mathbb{Z}$. This forces $d$ to be a multiple of $\dfrac{m}{m_0m_1}$. Hence we
                                             can always choose $0\leq
                                             d<\dfrac{m}{m_0m_1}$.
\end{proof}

\begin{lem} Given $K_0$-multiplicity $x$ and $p$ with $m$ dividing $p$, there exists a morphism
$\phi$ between $K_0(A_m,\mathbb{Z}_p)$ and $K_0(B_n,\mathbb{Z}_p)$,
such that $(x,\phi,0)$ is a triple in
$\textbf{Hom}(\textbf{K}(\textrm{A}_m;p),\textbf{K}(\textrm{B}_n;p)$.

\end{lem}

\begin{proof} Suppose $\phi=\left(
                              \begin{array}{cc}
                                \phi_{11} & \phi_{12} \\
                                \phi_{21} & \phi_{22} \\
                              \end{array}
                            \right),
$ as required, we need \begin{align} &\phi\left(
                                           \begin{array}{c}
                                             \bar{m}_0 \\
                                             \bar{m}_1 \\
                                           \end{array}
                                         \right)=\left(
                                                   \begin{array}{c}
                                                     x\bar{m}_0 \\
                                                     x\bar{m}_1 \\
                                                   \end{array}
                                                 \right)
                                         \\
&(-\dfrac{n}{pm_0},\dfrac{n}{pm_1})\phi\left(
                                    \begin{array}{c}
                                      \bar{0} \\
                                      \dfrac{p}{m}\bar{m}_1 \\
                                    \end{array}
                                  \right)=0.
\end{align} Expand these two equations, one get that \begin{align} &\phi_{11}m_0+\phi_{12}m_1\equiv xm_0 \,\text{mod}\,
p\\
&\phi_{21}m_0+\phi_{22}m_1\equiv xm_1 \,\text{mod}\, p,
\end{align} and \begin{align}&
\dfrac{n}{pm_1}\phi_{22}\dfrac{pm_1}{m}-\dfrac{n}{pm_0}\phi_{12}\dfrac{pm_1}{m}\equiv
0\,\text{mod}\,p\\
&\dfrac{n}{pm_1}\phi_{21}\dfrac{pm_0}{m}-\dfrac{n}{pm_0}\phi_{11}\dfrac{pm_0}{m}\equiv
0\,\text{mod}\,p.
\end{align} Namely, \begin{align}
&\dfrac{nm_1}{m}(\dfrac{\phi_{22}}{m_1}-\dfrac{\phi_{12}}{m_0})\equiv
0 \,\text{mod}\, p\\
&\dfrac{nm_0}{m}(\dfrac{\phi_{21}}{m_1}-\dfrac{\phi_{11}}{m_0})\equiv
0 \,\text{mod}\, p.
\end{align} Since we have $(m_0,m_1)=1$, there exists integer $\beta_0\geq
0$, and $\beta_1\leq 0$, such that $\beta_0m_0+\beta_1m_1=1$. Then
\begin{align}&xm_0\beta_0m_0+xm_0\beta_1m_1=xm_0\\
&xm_1\beta_0m_0+xm_1\beta_1m_1=xm_1.
\end{align}
Set $\phi_{11}=xm_0\beta_0$, $\phi_{12}=xm_0\beta_1$,
$\phi_{21}=xm_1\beta_0$, and $\phi_{22}=xm_1\beta_1$, then with
these data, (4.9) and (4.10) are satisfied.
\end{proof}

\begin{lem} Given $K_1$-multiplicity $y$ and $p$ with $m$ dividing $p$, there exists a morphism
$\psi$ between $K_0(A_m,\mathbb{Z}_p)$ and $K_0(B_n,\mathbb{Z}_p)$,
such that $(0,\psi,y)$ is a triple in
$\textbf{Hom}(\textbf{K}(\textrm{A}_m;p),\textbf{K}(\textrm{B}_n;p)$.

\end{lem}

\begin{proof} Assume $\psi=\left(
                             \begin{array}{cc}
                               \psi_{11} & \psi_{12} \\
                               \psi_{21} & \psi_{22} \\
                             \end{array}
                           \right),
$ as required, we need \begin{align}& \psi\left(
                                            \begin{array}{c}
                                              \bar{m}_0 \\
                                              \bar{m}_1 \\
                                            \end{array}
                                          \right)=0\\
                                          &(-\dfrac{n}{pm_0}, \dfrac{n}{pm_1})\psi\left(
                                                   \begin{array}{c}
                                                     \bar{0} \\
                                                     \dfrac{p}{m}\bar{m}_1 \\
                                                   \end{array}
                                                 \right)=y(-\dfrac{m}{pm_0}, \dfrac{m}{pm_1})\left(
                                                                                               \begin{array}{c}
                                                                                                 \bar{0} \\
                                                                                                 \dfrac{p}{m}\bar{m}_1 \\
                                                                                               \end{array}
                                                                                             \right).
\end{align} Expand these equations, we get \begin{align} &\psi_{11}m_0+\psi_{12}m_1\equiv 0 \,\text{mod}\,
p\\
&\psi_{21}m_0+\psi_{22}m_1\equiv 0 \,\text{mod}\, p,
\end{align} and \begin{align}
&\dfrac{nm_1}{m}(\dfrac{\psi_{22}}{m_1}-\dfrac{\psi_{12}}{m_0})\equiv
y \,\text{mod}\, p\\
&\dfrac{nm_0}{m}(\dfrac{\psi_{21}}{m_1}-\dfrac{\psi_{11}}{m_0})\equiv
-y \,\text{mod}\, p.
\end{align} Since we have $(m_0,m_1)=1$, there exists integer $\beta_0\geq
0$, and $\beta_1\leq 0$, such that $\beta_0m_0+\beta_1m_1=1$. Set
$$\psi=\left(
         \begin{array}{cc}
           \dfrac{mym_0m_1\beta_1}{nm_0} & -\dfrac{mym_0m_1\beta_1}{nm_1} \\
           -\dfrac{mym_0m_1\beta_0}{nm_0} & \dfrac{mym_0m_1\beta_0}{nm_1} \\
         \end{array}
       \right),
$$ then $\psi$ satisfies the requirement. (Note that $\dfrac{my}{n}$ is an integer, since $\dfrac{m}{m_0m_1}y$ is divided by $\dfrac{n}{m_0m_1}$,
the quotient is $\dfrac{my}{n}$.)

\end{proof}
\begin{thrm} \label{struc}(\textbf{Structure of} $\textbf{Hom}(\textbf{K}(\textrm{A}_m;p),\textbf{K}(\textrm{B}_n;p)$)
Any element $\Phi=(x, \rho, y)$ in
$\textbf{Hom}(\textbf{K}(\textrm{A}_m;p),\textbf{K}(\textrm{B}_n;p)$
with $K_0$-multiplicity $x$  and $K_1$-multiplicity $y$ is of the
following form:$$\Phi=(x,\sigma,y)+d(0,\left(
                               \begin{array}{cc}
                                 -m_1m_0 & m_0m_0 \\
                                 -m_1m_1 & m_0m_1 \\
                               \end{array}
                             \right), 0). \qquad{(\star)}$$ where $\sigma=\left(
                                                                  \begin{array}{cc}
                                                                    xm_0\beta_0+\dfrac{mym_1\beta_1}{n} & xm_0\beta_1-\dfrac{mym_0\beta_1}{n} \\
                                                                    xm_1\beta_0-\dfrac{mym_1\beta_0}{n} & xm_1\beta_1+\dfrac{mym_0\beta_0}{n} \\
                                                                  \end{array}
                                                                \right).$

\end{thrm}
\begin{proof} By Lemma 4.2 and 4.3, we obtain a triple $\alpha$ with
$K_0$-multiplicity $x$  and $K_1$-multiplicity $y$:$$\alpha=(x,
\sigma=\left(
                                                                  \begin{array}{cc}
                                                                    xm_0\beta_0+\dfrac{mym_1\beta_1}{n} & xm_0\beta_1-\dfrac{mym_0\beta_1}{n} \\
                                                                    xm_1\beta_0-\dfrac{mym_1\beta_0}{n} & xm_1\beta_1+\dfrac{mym_0\beta_0}{n} \\
                                                                  \end{array}
                                                                \right),
                                                                y
).$$ Suppose $\Phi=(x, \rho, y)$ is any element in
$\textbf{Hom}(\textbf{K}(\textrm{A}_m;p),\textbf{K}(\textrm{B}_n;p)$
with $K_0$-multiplicity $x$  and $K_1$-multiplicity $y$, then by
Lemma 4.1, we have $\Phi-\alpha=(0, \rho-\sigma, 0)$, so
$\rho-\sigma=d\left(
                               \begin{array}{cc}
                                 -m_1m_0 & m_0m_0 \\
                                 -m_1m_1 & m_0m_1 \\
                               \end{array}
                             \right)$. So $\Phi$ has the general
                             form: $$\Phi=(x,\sigma,y)+d(0,\left(
                               \begin{array}{cc}
                                 -m_1m_0 & m_0m_0 \\
                                 -m_1m_1 & m_0m_1 \\
                               \end{array}
                             \right), 0). \qquad{(\star)}$$
\end{proof} In the next, we calculate the induced maps of the four homomorphisms
mentioned in section 2. The following lemma is also basic in later
use.
\begin{lem}\label{induce map}
For any integers $n,m$ and $p$ with $m$ dividing $p$, the four
homomorphisms $\delta_0$, $\delta_1$, $id_{m,n}$, and
$\overline{id}_{m,n}$ induce the following triples: \begin{align*}
&\Gamma(\delta_0;p)=(m_0,\left(
                     \begin{array}{cc}
                       m_0 & 0 \\
                       m_1 & 0 \\
                     \end{array}
                   \right),
0),\\
& \Gamma(\delta_1;p)=(m_1,\left(
                     \begin{array}{cc}
                       0 & m_0 \\
                       0 & m_1 \\
                     \end{array}
                   \right),
0),\\
&\Gamma(id_{m,n};p)=(\frac{m}{(m,n)},\left(
                                 \begin{array}{cc}
                                   \frac{m}{(m,n)} & 0 \\
                                   0 & \frac{m}{(m,n)} \\
                                 \end{array}
                               \right),
\frac{n}{(m,n)}),\\
&\Gamma(\overline{id}_{m,n};p)=(\frac{m}{((m,n),\frac{m}{m_0m_1})},\left(
                                 \begin{array}{cc}
                                   0 & \frac{mm_0}{(m_1(m,n), \frac{m}{m_0})} \\
                                   \frac{mm_1}{(m_0(m,n), \frac{m}{m_1})} & 0 \\
                                 \end{array}
                               \right),
-\frac{n}{((m,n),\frac{m}{m_0m_1})}).
\end{align*}
\end{lem}

\begin{proof} To prove this lemma, we need to check the induced map
of these four homomorphisms on generators of K- theoretic groups. We
spell out the details for the convenience of readers.

For $\delta_0:  A_m\rightarrow B_n\otimes \textrm{M}_{m_0}$,
recalling the generator of $K_0(A_m)$ is the projection $h(t)$ with
$h(0)=1\otimes id_{m_0}\otimes id_{\frac{m}{m_0}}$, and
$h(1)=1\otimes id_{m_1}\otimes id_{\frac{m}{m_1}}$. Apply $\delta_0$
on $h(t)$, we get that $\delta_0(h)(0)=1\otimes id_{m_0}\otimes
id_{m_0}\otimes id_{\frac{n}{m_0}}$, and $\delta_0(h)(1)=1\otimes
id_{m_0}\otimes id_{m_1}\otimes id_{\frac{n}{m_1}}$. So the image is
$m_0$ times the generator of $K_0(B_n)$. So the $K_0$-multiplicity
of $\Gamma(\delta_0;p)$ is $m_0$. Similarly, the $K_0$-multiplicity
of $\Gamma(\delta_1;p)$ is $m_1$. Moreover, it is obvious that the
$K_1$-multiplicity of both $\Gamma(\delta_0;p)$ and
$\Gamma(\delta_1;p)$ are zero.

For the part of $\Gamma(\delta_0;p)$ on $K_0(A_m,\mathbb{Z}_p)$,
note that the restriction of $id: \tilde{I}_p\rightarrow A_m$ and
$\delta_1: \tilde{I}_p\rightarrow A_m$ are the generators of
$K_0(A_m,\mathbb{Z}_p)$. Apply $\delta_0$ on them, by our
identification, we get \begin{align*}&\Gamma(\delta_0;p)\left(
                                                          \begin{array}{c}
                                                            \bar{0} \\
                                                            \dfrac{p}{m}\bar{m}_1 \\
                                                          \end{array}
                                                        \right)=\left(
                                                                  \begin{array}{c}
                                                                    \bar{0} \\
                                                                    \bar{0} \\
                                                                  \end{array}
                                                                \right)\\
                &\Gamma(\delta_0;p)\left(
                                     \begin{array}{c}
                                       \bar{m}_0 \\
                                       \bar{m}_1 \\
                                     \end{array}
                                   \right)=\left(
                                             \begin{array}{c}
                                               \overline{m^2_0} \\
                                               \overline{m_0m_1} \\
                                             \end{array}
                                           \right)
\end{align*} which indicates that the middle part of $\Gamma(\delta_0;p)$ is $\left(
      \begin{array}{cc}
        m_0 & 0 \\
        m_1 & 0 \\
      \end{array}
    \right).
$ Similarly, we can get $\Gamma(\delta_1;p)$.

For $id_{m,n}: A_m\rightarrow B_n\otimes
\textrm{M}_{\frac{m}{(m,n)}}$, it is easily seen that the
$K_0$-multiplicity of $\Gamma(id_{m,n};p)$ is $\dfrac{m}{(m,n)}$.
But for the $K_0$-multiplicity of $\Gamma(\overline{id}_{m,n})$, we
need to be careful (the different dimension drops cause non
symmetric size compare with $id_{m,n}$): the $K_0$-multiplicity of
$\Gamma(\overline{id}_{m,n})$ is
$\dfrac{m}{((m,n),\dfrac{m}{m_0m_1})}$. The $K_1$-multiplicity of
$\Gamma(id_{m,n};p)$ is $\dfrac{n}{(m,n)}$; The $K_1$-multiplicity
of $\Gamma(\overline{id}_{m,n})$ is
$-\dfrac{n}{((m,n),\dfrac{m}{m_0m_1})}$.

For the middle part of $\Gamma(id_{m,n};p)$ and
$\Gamma(\overline{id}_{m,n};p)$, we also go through the compositions
on generators, and the identification in Theorem \ref{calcu}. it is
straightforward but tedious, then we get the formulas.

\end{proof}

\begin{prop} $KK(A_m,B_n)\cong
\mathbb{Z}\oplus\mathbb{Z}_{(n,m)}\oplus\mathbb{Z}_{\frac{m}{m_0m_1}}$,
and $\delta_0, \delta_1$ and $id_{m,n}$ generate $KK(A_m,B_n)$.
\end{prop}
\begin{proof} By UCT, it is easily seen that $KK(A_m,B_n)\cong
\mathbb{Z}\oplus\mathbb{Z}_{(n,m)}\oplus\mathbb{Z}_m$, moreover,
$\mathbb{Z}=Hom(K_0(A_m),K_0(B_n))$,
$\mathbb{Z}_{(n,m)}=Hom(K_1(A_m),K_1(B_n))$, and
$Ext(K_1(A_m),K_0(B_n))=\mathbb{Z}_{\frac{m}{m_0m_1}}$.

By Lemma \ref{induce map}, $\mathbb{Z}$ is generated by $\delta_0$,
and $\mathbb{Z}_{(n,m)}$ is generated by
$id_{m,n}-\dfrac{m}{(m,n)}(\beta_0\delta_0+\beta_1\delta_1)$, where
$\beta_0,\beta_1$ are integers such that $\beta_0m_0+\beta_1m_1=1$.
On the other hand, the non zero KK-element $m_1\delta_0-m_0\delta_1$
induces trivial map on K-groups, hence it gives an nontrivial
extension. Since $m_1\delta_0-m_0\delta_1$ has order
$\frac{m}{m_0m_1}$, then we are done.
\end{proof}

\begin{thrm} Given positive integers $n,m$ and $p$ with $n,m|p$, then
the canonical map $$\Gamma: KK(A_m,B_n)\rightarrow
\textbf{Hom}(\textbf{K}(\textrm{A}_m;p),\textbf{K}(\textrm{B}_n;p)$$
is an isomorphism.
\end{thrm}
\begin{proof} Given any element $\Phi=(x,\tau, y)$ in
$\textbf{Hom}(\textbf{K}(\textrm{A}_m;p),\textbf{K}(\textrm{B}_n;p)$,
by UCT, we know that there is a KK-element $\alpha$, such that
$\Gamma(\alpha;p)=(x,\eta,y)$ in
$\textbf{Hom}(\textbf{K}(\textrm{A}_m;p),\textbf{K}(\textrm{B}_n;p)$.
By Lemma 4.1, $$\Phi-\Gamma(\alpha;p)=(0,\tau-\eta,0)=d(0, \left(
                                                 \begin{array}{cc}
                                                   -m_1m_0 & m_0m_0 \\
                                                   -m_1m_1 & m_0m_1 \\
                                                 \end{array}
                                               \right)
, 0)$$ for some integer $d$ with $0\leq d<\dfrac{m}{m_0m_1}$. By
Lemma \ref{induce map}, we have that $$d(0, \left(
                                                 \begin{array}{cc}
                                                   -m_1m_0 & m_0m_0 \\
                                                   -m_1m_1 & m_0m_1 \\
                                                 \end{array}
                                               \right)
, 0)=dm_0\Gamma(\delta_1;p)-dm_1\Gamma(\delta_0;p).$$ Hence,
$$\Phi=\Gamma(\alpha;p)+dm_0\Gamma(\delta_1;p)-dm_1\Gamma(\delta_0;p).$$
So, $\Gamma$ is surjective.

By Proposition 4.6, we know that $KK(A_m,B_n)$ has
$(m,n)\dfrac{m}{m_0m_1}$ torsion elements; on the other hand,
$\textbf{Hom}(\textbf{K}(\textrm{A}_m;p),\textbf{K}(\textrm{B}_n;p)$
has at least $(m,n)\dfrac{m}{m_0m_1}$ torsion elements by Lemma 4.1
(we have the assumption $m|p$ there). Since any surjective morphism
from $\mathbb{Z}$ to $\mathbb{Z}$ is automatically injective, so the
map $\Gamma$ on torsion part must be injective (since we have shown
$\Gamma$ is surjective above). Therefore $\Gamma$ is an isomorphism.

\end{proof}

Now, we can have a better understanding about the structure of
$\textbf{Hom}(\textbf{K}(\textrm{A}_m;p),\textbf{K}(\textrm{B}_n;p)$,
the torsion elements $$d(0, \left(
                                                 \begin{array}{cc}
                                                   -m_1m_0 & m_0m_0 \\
                                                   -m_1m_1 & m_0m_1 \\
                                                 \end{array}
                                               \right)
, 0)$$ are exactly $dm_0\Gamma(\delta_1;p)-dm_1\Gamma(\delta_0;p)$
by Lemma \ref{induce map}. Hence, for any triple $\alpha=(x, \omega,
0)$ with $K_1$-multiplicity zero, by Theorem \ref{struc}, we know it
is of the following form: $$\alpha=(x, \left(
                                         \begin{array}{cc}
                                           xm_0\beta_0 & xm_0\beta_1 \\
                                           xm_1\beta_0 & xm_1\beta_1 \\
                                         \end{array}
                                       \right),
0)+d(0, \left(
                                                 \begin{array}{cc}
                                                   -m_1m_0 & m_0m_0 \\
                                                   -m_1m_1 & m_0m_1 \\
                                                 \end{array}
                                               \right)
, 0).$$ Where $\beta_0m_0+\beta_1m_1=1, \beta\geq 0, \beta_1\leq 0$.
By Lemma \ref{induce map} again, we get
\begin{align*}\alpha &=\beta_0x\Gamma(\delta_0;p)+\beta_1x\Gamma(\delta_1;p)+dm_0\Gamma(\delta_1;p)-dm_1\Gamma(\delta_0;p)\\
&=(\beta_0x-dm_1)\Gamma(\delta_0;p)+(\beta_1x+dm_0)\Gamma(\delta_1;p).\qquad{(\dagger)}
\end{align*}
For a general triple $\Phi=(x,\sigma,y)$ in Theorem \ref{struc},
first, there is an integer $k$, such that $y=k\dfrac{n}{(m,n)}$,
then $\dfrac{my}{n}=k\dfrac{m}{(m,n)}$. Therefore,
$$k\Gamma(id_{m,n};p)=(\dfrac{my}{n}, \left(
                                       \begin{array}{cc}
                                         \dfrac{my}{n} & 0 \\
                                         0 & \dfrac{my}{n} \\
                                       \end{array}
                                     \right),
y).$$ Then it is straightforward to verify that
$$\Phi-k\Gamma(id_{m,n};p)=((x-\frac{my}{n})\beta_0-dm_1)\Gamma(\delta_0;p)+((x-\frac{my}{n})\beta_1+dm_0)\Gamma(\delta_1;p).$$
Then
$$\Phi=k\Gamma(id_{m,n};p)+((x-\frac{my}{n})\beta_0-dm_1)\Gamma(\delta_0;p)+((x-\frac{my}{n})\beta_1+dm_0)\Gamma(\delta_1;p).$$
So in K-theory with coefficient picture, KK-group is also generated
by $\delta_0, \delta_1$ and $id_{m,n}$.

\section{KK-lifting problem for generalized dimension drop interval algebras}
In this section, we start to investigate the KK-lifting problem for
generalized dimension drop interval algebras. Set $KK^{+}(A_m,B_n)$
to be following set:
$$\{\kappa\in KK(A_m,B_n)\,|\,\kappa=[\varphi] \,\text{for some}\,\varphi\in
Hom(A_m,M_k(B_n))\}.$$ Then we try to find some conditions under
which a KK-element can lies in this set; moreover, try to relate
these conditions to proper invariants of C*-algebras such that it is
applicable to the classification program.

In \cite{JS}, X. Jiang and H. Su investigated these kind building
blocks , they already had a criterion for KK-lifting in terms of
K-homology. We first recall the necessary preliminaries here.

\begin{lem} (Lemma 3.1 in \cite{JS}) Given $A_m=I[m_0,m,m_1]$, then its K-homology group is
generated by the two irreducible representations $V_0, V_1$ up to
the following relation:$$\dfrac{m}{m_0}[V_0]=\dfrac{m}{m_1}[V_1].$$
\end{lem}

\begin{remark}\label{relation} For $\delta_0,\delta_1\in KK(A_m,B_n)$, it is straightforward to see
that \begin{align*} \delta_0([V^{B_n}_{0}])=m_0[V^{A_m}_0],
\delta_0([V^{B_n}_{1}])=m_1[V^{A_m}_0]\\
\delta_1([V^{B_n}_{0}])=m_0[V^{A_m}_1],
\delta_1([V^{B_n}_{1}])=m_1[V^{A_m}_1]
\end{align*} By this lemma, the relation between $\delta_0$, and $\delta_1$
is $\dfrac{m}{m_0}\delta_0=\dfrac{m}{m_1}\delta_1$.
\end{remark}
Moreover, they defined an order structure on the K-homology
groups:$$K^{0}_{+}(A_m)\triangleq\{[\rho]\!\in\! K^{0}(A_m)\,|\,\rho
\,\text{is a finite dimensional representation of }\,A_m\}.$$ Then
they proved the following criterion for KK-lifting.
\begin{thrm} (Theorem 3.7 in \cite{JS}) Given $\alpha\in
KK(A_m,B_n)$, then $\alpha$ can be lifted to a $*$-homomorphism if
and only if $\alpha^*$ is positive from $K^{0}(B_n)$ to
$K^{0}(A_m)$, where $\alpha^*$ is the Kasparov product with
K-homology groups.
\end{thrm}

\begin{remark} 1. Moreover, Jiang and Su proved that $KK(A_m,B_n)\cong Hom(K^{0}(B_n),K^{0}(A_m))$ under the Kasparov
product. Hence, the isomorphism is an ordered isomorphism.

 2.We change a little bit the original statement of
Jiang and Su's theorem, since we don't require the homomorphism to
be unital, see the remark after Jiang and Su's Theorem 3.7 in
\cite{JS}.
\end{remark}
On the other hand, the Dadarlat-Loring order on K-theory with
coefficient gave a lifting criterion for KK-elements between
classical dimension drop interval algebras (see Proposition 3.2 in
\cite{Ei}), and succeeded as an invariant for the classification of
real rank zero limits of these classical ones (see \cite{DL1},
\cite{DL2}, \cite{Ei}). So it is natural to ask for the
corresponding KK-lifting criterion for generalized dimension drop
interval algebras in Dadarlat-Loring's picture. Let us try to do
this.

For simplicity, we first look at the KK-elements with zero
$K_1$-multiplicity, realize them on
$\textbf{Hom}(\textbf{K}(\textrm{A}_m;p),\textbf{K}(\textrm{B}_n;p)$,
they all have the form as in $(\dagger)$, then by Theorem 4.7, these
KK-elements are of the form
$$\alpha=(\beta_0x-dm_1)\delta_0+(\beta_1x+dm_0)\delta_1, 0\leq d<\dfrac{m}{m_0m_1}.\qquad{(**)}$$

To relate the lifting of these elements to K-theory with
coefficient, we first have the following proposition.

\begin{prop} Given any KK-element $\alpha\in KK(A_m,B_n)$ with $K_1$-multiplicity
zero as in $(**)$ with $0\leq d<\dfrac{m}{m_0m_1m_0}$ or $0\leq
d<\dfrac{m}{m_0m_1m_1}$, given any $p$ with $m|p$, if
$\Gamma(\alpha;p)$ preserves the Dadarlat-Loring order, then
$\beta_0x-dm_1\geq0$.

\end{prop}

\begin{proof} We know that $$\Gamma(\alpha;p)=(x, \left(
                                         \begin{array}{cc}
                                           xm_0\beta_0 & xm_0\beta_1 \\
                                           xm_1\beta_0 & xm_1\beta_1 \\
                                         \end{array}
                                       \right),
0)+d(0, \left(
                                                 \begin{array}{cc}
                                                   -m_1m_0 & m_0m_0 \\
                                                   -m_1m_1 & m_0m_1 \\
                                                 \end{array}
                                               \right)
, 0).$$ Since we always assume $\beta_0\geq0, \beta_1\leq0$, if
$\alpha$ preserves the Dadarlat-Loring order, then $x\geq0$. So
$-x\beta_1\Gamma(\delta_1;p)$ preserves the Dadarlat-Loring order,
so $\Gamma(\alpha;p)-x\beta_1\Gamma(\delta_1;p)$ also preserves this
order.

Note that $$-x\beta_1\Gamma(\delta_1;p)=(-x\beta_1m_1, \left(
                                                                \begin{array}{cc}
                                                                  0 & -x\beta_1m_0 \\
                                                                  0 & -x\beta_1m_1 \\
                                                                \end{array}
                                                              \right),
0).$$ Hence,
\begin{align*}\Gamma(\alpha;p)-x\beta_1\Gamma(\delta_1;p)&=(x(1-\beta_1m_1), \left(
                                                                                \begin{array}{cc}
                                                                                  xm_0\beta_0-dm_1m_0 & dm_0m_1 \\
                                                                                  xm_1\beta_0-dm_1m_1 & dm_0m_1 \\
                                                                                \end{array}
                                                                              \right)
 , 0)\\
 &=(x\beta_0m_0, \left(
                                                                                \begin{array}{cc}
                                                                                  (x\beta_0-dm_1)m_0 & dm_0m_0 \\
                                                                                  (x\beta_0-dm_1)m_1 & dm_0m_1 \\
                                                                                \end{array}
                                                                              \right) , 0)
\end{align*} Because this is positive, then the image of $(\dfrac{p}{m}, \bar{0},
\dfrac{p}{m}\bar{m_1})\in K_0(A_m;G_p)$ under this triple is in the
positive cone of $K_0(B_n;G_p)$. This means that
$$\dfrac{px\beta_0m_0}{m}m_0\geq\overline{\dfrac{pm_1}{m}dm_0m_0},\, \text{and}\, \dfrac{px\beta_0m_0}{m}m_1\geq\overline{\dfrac{pm_1}{m}dm_0m_1}.$$
Since $0\leq d<\dfrac{m}{m_0m_1m_0}$ or $0\leq
d<\dfrac{m}{m_0m_1m_0}$, then we have $\dfrac{pm_1}{m}dm_0m_0<p$ or
$\dfrac{pm_1}{m}dm_0m_1<p$. Therefore, we have either
$\dfrac{px\beta_0m_0}{m}m_0\geq \dfrac{pm_1}{m}dm_0m_0$ or
$\dfrac{px\beta_0m_0}{m}m_1\geq \dfrac{pm_1}{m}dm_0m_1$, in any
case, we get $\beta_0x-dm_1\geq 0$.
\end{proof}

\begin{remark} For the classical dimension drop interval algebras, we can choose $\beta_0=1, \beta_1=0$, then $(**)$ becomes
$$\alpha=(x-d)\delta_0+d\delta_1.$$ Hence $\alpha$ is positive in Dadarlat-Loring's sense if and only if it can be lifted to a homomorphism.
This proposition is then exactly Lemma 3.1 in \cite{Ei}. But now we
could have $\beta_1\leq0$, this indicates that the Dadarlat-Loring
order may fail to guarantee the lifting of $\alpha$.
\end{remark}

\begin{prop} Given any KK-element $\alpha\in KK(A_m,B_n)$ with $K_1$-multiplicity
zero as in $(**)$, if the $K_0$-multiplicity $x\geq m$, then
$\alpha$ can be lifted to a homomorphism between the algebras.
\end{prop}

\begin{proof} By Remark \ref{relation}, to determine whether
$\alpha\in(**)$ can be lifted, it is equivalent to count the numbers
of $\delta_0$ and $\delta_1$. If $x\geq m$, then
$\beta_0x-dm_1\geq0$, and $\beta_1x+dm_0\leq0$. Assume that
$\beta_0x-dm_1=\dfrac{m}{m_0}j_0+r_0, 0\leq r_0<\dfrac{m}{m_0}$, and
$|\beta_1x+dm_0|=\dfrac{m}{m_1}j_1+r_1, 0\leq r_1<\dfrac{m}{m_1}$.
Then
\begin{align*}\alpha&=j_0\dfrac{m}{m_0}\delta_0+r_0\delta_0-(j_1\dfrac{m}{m_1}\delta_1+r_1\delta_1)\\
&=((j_0-j_1)\dfrac{m}{m_1}-r_1)\delta_1+r_0\delta_0
\end{align*} While
\begin{align*}j_0-j_1&=\dfrac{\beta_0x-dm_1-r_0}{\dfrac{m}{m_0}}-\dfrac{|\beta_1x+dm_0|-r_1}{\dfrac{m}{m_1}}\\
&=\dfrac{x-(r_0m_0-r_1m_1)}{m}
\end{align*} Since $x\geq m$, we have that $j_0-j_1\geq 1$. So $\alpha$ can be lifted.
\end{proof}

To make situation easier, we assume $d=0$, we try to investigate the
exact conditions which imply $\alpha$ in $(**)$ preserves the
Dadarlat-Loring order. We have the following proposition.

\begin{prop} Given $\alpha=\beta_0x\delta_0+\beta_1x\delta_1$, let $R$ be the remainder of $\beta_0m_0m_0x$ divided by $m$
, and $S$ be the remainder of $\beta_0m_0m_1x$ divided by $m$. Then
$\Gamma(\alpha;p)$ preserves the Dadarlat-Loring order structure if
and only if $x=0$ or
\begin{align} \beta_0m_0m_0x\geq m,\, \beta_0m_0m_1x\geq m\\
m_0x\geq R,\, m_1x\geq S.
\end{align}
\end{prop}

\begin{proof} To determine whether $\Gamma(\alpha;p)$ preserves the
Dadarlat-Loring order, by Lemma \ref{lem:generator of cone}, we only
need to work on the generators of positive cone: \begin{align*}
&[\delta_0]=(1,\bar{0},\bar{0}), [\delta_1]=(1,\bar{m}_0,\bar{m}_1)\\
&[id]=(\dfrac{p}{m},\bar{0},\dfrac{p}{m}\bar{m}_1),
[\overline{id}]=(\dfrac{p}{m},\dfrac{p}{m}\bar{m}_0,\bar{0})
\end{align*}

1. The image of $[\delta_0]$ is automatically positive.

2. For the image of $[\delta_1]$,
\begin{align*}
\Gamma(\alpha;p)([\delta_1])&=(x, \left(
                                                      \begin{array}{cc}
                                                        \beta_0m_0x & \beta_1m_0x \\
                                                        \beta_0m_1x & \beta_1m_1x \\
                                                      \end{array}
                                                    \right)\left(
                                                             \begin{array}{c}
                                                               \bar{m}_0 \\
                                                               \bar{m}_1 \\
                                                             \end{array}
                                                           \right)
                                                    ,
0)\\
&=(x, \left(
        \begin{array}{c}
          \bar{m}_0x \\
          \bar{m}_1x \\
        \end{array}
      \right)
0),
\end{align*} This is also always positive.

3. The image of $[\overline{id}]$ is
\begin{align*}\Gamma(\alpha;p)([\overline{id}])&=(x\dfrac{p}{m}, \left(
                                                      \begin{array}{cc}
                                                        \beta_0m_0x & \beta_1m_0x \\
                                                        \beta_0m_1x & \beta_1m_1x \\
                                                      \end{array}
                                                    \right)\left(
                                                             \begin{array}{c}
                                                               \dfrac{p}{m}\bar{m}_0 \\
                                                               \bar{0} \\
                                                             \end{array}
                                                           \right)
                                                    ,
0)\\
&=(x\dfrac{p}{m}, \left(
        \begin{array}{c}
          \bar{m}_0x\beta_0m_0\dfrac{p}{m} \\
          \bar{m}_0x\beta_0m_1\dfrac{p}{m}\\
        \end{array}
      \right)
0).
\end{align*} Positivity means that $$x\dfrac{p}{m}m_0\geq \bar{m}_0x\beta_0m_0\dfrac{p}{m},\; x\dfrac{p}{m}m_1\geq
\bar{m}_0x\beta_0m_1\dfrac{p}{m}.$$ These force first that
\begin{align*} \beta_0m_0m_0x\geq m,\; \beta_0m_0m_1x\geq m.
\end{align*} Moreover, write \begin{align*}\dfrac{\beta_0m_0m_0x}{m}=\llcorner\dfrac{\beta_0m_0m_0x}{m}\lrcorner+r, 0\leq
r<1\\
\dfrac{\beta_0m_0m_1x}{m}=\llcorner\dfrac{\beta_0m_0m_1x}{m}\lrcorner+s,
0\leq s<1
\end{align*}, then positivity is equivalent to \begin{align*}
\overline{rp}\leq \dfrac{pm_0x}{m}\quad\overline{sp}\leq
\dfrac{pm_1x}{m},
\end{align*} which is \begin{align*}
R=rm\leq m_0x\quad S=sm\leq m_1x.
\end{align*}

4. The image of $[id]$ is
\begin{align*}\Gamma(\alpha;p)([id])&=(x\dfrac{p}{m}, \left(
                                                      \begin{array}{cc}
                                                        \beta_0m_0x & \beta_1m_0x \\
                                                        \beta_0m_1x & \beta_1m_1x \\
                                                      \end{array}
                                                    \right)\left(
                                                             \begin{array}{c}
                                                               \bar{0} \\
                                                               \dfrac{p}{m}\bar{m}_1 \\
                                                             \end{array}
                                                           \right)
                                                    ,
0)\\
&=(x\dfrac{p}{m}, \left(
        \begin{array}{c}
          \bar{m}_1x\beta_1m_0\dfrac{p}{m} \\
          \bar{m}_1x\beta_1m_1\dfrac{p}{m}\\
        \end{array}
      \right)
0).
\end{align*}
Positivity means that $$x\dfrac{p}{m}m_0\geq
\bar{m}_1x\beta_1m_0\dfrac{p}{m},\; x\dfrac{p}{m}m_1\geq
\bar{m}_1x\beta_1m_1\dfrac{p}{m}.$$ Note that $\beta_1\leq0$, we
need a little more work. Since
$\beta_1m_1m_0x\dfrac{p}{m}+(\beta_0m_0-1)m_0x\dfrac{p}{m}=0$, so
\begin{align*}
\bar{m}_1x\beta_1m_0\dfrac{p}{m}&=(-\beta_0m_0+1)\dfrac{p}{m}x\bar{m}_0,\\
&=\bar{m}_0\dfrac{p}{m}x-\bar{m}_0\beta_0m_0\dfrac{p}{m}\\
&=\bar{m}_0\dfrac{p}{m}x-\overline{rp}\\
&=\overline{m_0\dfrac{p}{m}x-rp}
\end{align*}  The condition $R\leq m_0x$ is equivalent to $m_0\dfrac{p}{m}x-rp\geq0$. Hence, $x\dfrac{p}{m}m_0\geq
\bar{m}_1x\beta_1m_0\dfrac{p}{m}$. Similarly, with the condition
$S\leq m_1x$ we have $x\dfrac{p}{m}m_1\geq
\bar{m}_1x\beta_1m_1\dfrac{p}{m}$.
\end{proof}

\begin{remark} In the classical dimension drop algebra case, the conditions above becomes $x\geq m$, and this is enough to guarantee a lifting.
So for general case, the non-positive number $\beta_1$ caused by
different dimension drops at two endpoints really gives us the
possibility for counterexamples.

\end{remark}

Now, we are able to prove the main Theorem \ref{main}, namely, give
examples of KK-elements which preserve the Dadarlat-Loring order
structure, and fail to be lifted to a $*$-homomorphism.

\begin{proof}[Proof of Theorem \ref{main}]  Let $m_0=2, m_1=3, m=12$, then we take $\beta_0=2,
\beta_1=-1$, we want some KK-elements $\alpha=2x\delta_0-x\delta_1$,
such that $\alpha$ preserves the Dadarlat-Loring order structure but
fail to be lifted to a $*$-homomorphism. By Proposition 5.8, we need
the $K_0$-multiplicity $x$ satisfies the following
inequalities: \begin{align}8x\geq 12, 12x\geq12\\
2x\geq R, 3x\geq S
\end{align} From (5.3), we get $x\geq2$, take $x=2$, then (5.2) is
satisfied, then $\alpha=4\delta_0-2\delta_1$ can not be lifted to a
homomorphism by Jiang and Su's criterion. By Proposition 5.7, such
examples only exist for $x<m$, and under our algorithm, we can
actually determine all the possible examples. If $x=5$, we get
$10\delta_0-5\delta_1=4\delta_0-\delta_1$, which also fits our
purpose. In fact, $x=3$ and $x=5$ are all the possibilities.
\end{proof}
\begin{remark} 1. Let us summarize the KK-lifting story since
Elliott's paper \cite{Ell}, he defined an order structure which is
good for KK-lifting of circle algebras, then S. Eilers found
counterexamples for classical dimension drop algebras, i.e.,
$\delta_0-\delta_1$ (with symmetric coefficients not exceed the
generic size), then Dadarlat and Loring defined an order structure
on the K-theory with coefficient, which can kill the counterexample
$\delta_0-\delta_1$. Now, we find other counterexamples for
generalized dimension drop algebras, namely, certain linear
combination of $\delta_0, \delta_1$ with non symmetric sizes, a
natural question is do we have a new order structure to exclude
these ones. This could be answered in another paper.

2. In the present paper, we show the existence theorem fails at
building block level (as for generalized dimension drop interval
algebras), but we didn't investigate it for limit algebras. For
simple limits, as it is well known, the $K_0$-multiplicities of any
partial map would be arbitrarily large, then no trouble. For real
rank zero case, we can still get a classification by the
Dadarlat-Loring order structure, because real rank zero condition
can more or less control the dynamical behavior of connecting maps.
This is done in a forthcoming paper. However, the most general
limits are different story.

3. Jiang and Su's criterion for KK-lifting is useful for more
general C*-algebras on interval with dimension drops, e.g. splitting
interval algebras with dimension drops in \cite{Li}.

\end{remark}

\proof[Acknowledgements]This paper was completed while the second
author was a postdoctoral fellow at the Fields Institute (and
University of Toronto). We wish to thank the Fields Institute for
its hospitality.

\end{document}